\documentclass[12pt, reqno]{amsart}
\usepackage{amssymb,amsthm,amsmath}
\usepackage[numbers,sort&compress]{natbib}
\usepackage{amsfonts}
\usepackage{mathrsfs}
\usepackage{latexsym}
\usepackage{tikz}
\usepackage{tkz-euclide}
\usepackage{float}
\usepackage{pdfsync}
\usepackage{indentfirst}
\usepackage[colorlinks=true, linkcolor=blue, citecolor=red]{hyperref} 

\newcommand{\beq}{\begin{equation}}
	\newcommand{\eeq}{\end{equation}}
\newcommand{\ben}{\begin{eqnarray}}
	\newcommand{\een}{\end{eqnarray}}
\newcommand{\beno}{\begin{eqnarray*}}
	\newcommand{\eeno}{\end{eqnarray*}}
\numberwithin{equation}{section} 
\newtheorem{Thm}{Theorem}[section] 
\newtheorem{Lem}[Thm]{Lemma} 
\newtheorem{Cor}[Thm]{Corollary}
\newtheorem{Rem}[Thm]{Remark}
\newtheorem{Def}[Thm]{Definition}

\numberwithin{equation}{section}
\newcommand{\no}{\nonumber}
\newcommand{\m}{\mkern-10mu}

\allowdisplaybreaks

\begin{document}
	\title[Liouville type theorems for the fractional Navier-Stokes equations]{\bf Liouville type theorems for the fractional Navier-Stokes equations without the integrability condition of velocity in $\mathbb{R}^3$}
	
	\author{Wendong~Wang}
	\address[Wendong~Wang]{School of Mathematical Sciences, Dalian University of Technology, Dalian, 116024,  China}
	\email{wendong@dlut.edu.cn}
	
	\author{Guoxu~Yang*} \thanks{*Author to whom any correspondence should be addressed.}
	\address[Guoxu~Yang]{School of Mathematical Sciences, Dalian University of Technology, Dalian, 116024,  China}
	\email{guoxu\_dlut@outlook.com}

	\author{Jianbo~Yu}
	\address[Jianbo~Yu]{School of Mathematical Sciences, Dalian University of Technology, Dalian, 116024,  China}
	\email{yujb@mail.dlut.edu.cn}
	
	\date{\today}


	\begin{abstract}
Motivated by the classification of solutions of harmonic functions,
 we investigate Liouville type theorems for the fractional Navier-Stokes equations  in $\mathbb{R}^3$ under some conditions on the boundedness of fractional derivatives. We prove that the smooth solution must be a trivial solution provided that it uniformly converges to a nonzero constant vector at infinity by applying Lizorkin's multiplier theorem to establish \(L^p\) estimates for the fractional linear Oseen system and Coifman-McIntosh-Meyer type commutator estimates for the dissipation term. It is noteworthy that the integrability of velocity is not required here. 
	\end{abstract}
	
	\maketitle
	

	
	
	\section{Introduction}
    Consider the following steady-state fractional Navier-Stokes equations in $\mathbb{R}^3$ :
	\begin{align} \label{eq:main}
	\left\{\begin{array}{l}
		-(-\Delta)^s u=u \cdot \nabla u+\nabla p, \\
		\operatorname{div} u=0,
	\end{array}\right.
    \end{align}
   where $u=\left(u_1(x), u_2(x), u_3(x)\right)$ is a vector field and $p=p(x)$ is a scalar field. As in \cite{AVMRTM10, Kw2017}, the fractional Laplacian of a function $f: \mathbb{R}^N \rightarrow \mathbb{R}$ is expressed by the formula
	$$
	(-\Delta)^{s} f(x):=C_{N, s} \int_{\mathbb{R}^N} \frac{f(x)-f(y)}{|x-y|^{N+2s}} d y,
	$$
	where the parameter $s$ is a real number with $0<s < 1$ and $C_{N, s}$ is a normalization constant which is given by
	$$
	C_{N, s}=\frac{\Gamma\left(\frac{N+2s}{2}\right)}{2^{-2s} \pi^{\frac{N}{2}}\left|\Gamma\left(-s\right)\right|}.
	$$
	The solutions to the fractional diffusion equations are L\'{e}vy motions, a generalization of Brownian motion using $s/2$-stable distributions, and they are scaling limits of random walks with power-law transition probabilities, and their sample paths are random fractals whose dimension $s/2$ coincides with the order of the fractional derivative (see  \cite{ST1994}). Experimentally, the fractional Laplacian operator $(-\Delta)^{s}$ has proven effective in simulating non-classical reaction-diffusion dynamics within porous media systems (see \cite{MDB1999, MDB2001}) and in computational turbulence models (see Chapter 13.2 in \cite{P2000}).

%

    We are interested in the classification of solutions of the system \eqref{eq:main}. A classical result asserts that if the $k$-th gradient of a harmonic function is bounded, then it is a $k$-th polynomial; see also the work of Yau \cite{Yau} and Li-Tam \cite{Li-Tam}, where they considered the space of harmonic functions on complete manifold with
    nonnegative Ricci curvature with linear growth. Hence, a natural question is 
    \begin{quote}
    \centering{\it Does  $u\equiv C$ holds if $u$ is bounded in $\mathbb{R}^3$?}
    \end{quote}
    Even for $s=1$, i.e., the steady Navier-Stokes equations
\begin{align} \label{eq:NS}
	\left\{\begin{array}{l}
		-\Delta u+u \cdot \nabla u+\nabla p=0, \\
		\operatorname{div} u=0,
	\end{array}\right.
\end{align}
this is not solved. As said by Koch-Nadirashvili-Seregin-Sverak in \cite{KNSS},  
``The case of general 3-dimensional fields is, as far as we know, completely open. In fact, it is open even in the steady-state case ($u$ is independent of $t$)." For the two dimensional case, bounded solutions for the steady Navier-Stokes equations are constants proved by Gilbarg-Weinberger \cite{GW1978}, where they made use of the fact that the vorticity function satisfies
a nice elliptic equation to which the maximum principle applies; see also the same result in \cite{KNSS} as a byproduct of their
work on the non-stationary case. For more references, we refer to  \cite{BFZ2013} for the unbounded velocity, \cite{WW2023} for the classification of solutions in cone and the references therein.
For the three dimensional case, it is also very difficult, even if we assume the Dirichlet integral \(\int_{\mathbb{R}^3} |\nabla u|^2 \, dx\) is finite, which is so-called D-solution, and $u$ is vanishing at infinity.  This interesting question originates from Leray's seminal work \cite{L1933} and is explicitly discussed in Galdi's book (Remark X.9.4, \cite{Galdi}; see also Conjecture 2.5 in Tsai's book \cite{Tai2018}). The Liouville type problem, without any additional assumptions, remains widely open. Galdi proved a Liouville type theorem in \cite{Galdi} under the condition \(u \in L^{\frac{9}{2}}(\mathbb{R}^3)\). Chae \cite{Chae2014} showed that the condition \(\Delta u \in L^{\frac{6}{5}}(\mathbb{R}^3)\) is sufficient for \(u\) to vanish by using the maximum principle for the head pressure.  Seregin further established the criterion \(u \in BMO^{-1}(\mathbb{R}^3)\) in \cite{Se2016}. For more results on conditional Liouville properties, we refer to \cite{CW2016, Chae2021,CW2019,KTW2021,SW2019} and the references therein.

When we do not assume $u$ is vanishing at infinity,
Finn \cite{Fin} and Ladyzhenskaya (See Theorem 8, Chapter 5 in \cite{Lady}) showed that any D-solution of Navier-Stokes equations \eqref{eq:NS} in a 3D exterior domain converges to a prescribed constant vector $u_{\infty}$ at infinity (the derivatives of the velocity and the pressure also converge to the constant in 3D, for example seeing Theorem X.5.1 in \cite{Galdi}), i.e.,
\begin{eqnarray*} \label{eq:the uniform condition}
	u(x) \rightarrow u_\infty \quad \text { uniformly as } \quad|x| \rightarrow \infty.
\end{eqnarray*}
When $ u_\infty \neq 0$, one could obtain $u\in L^3(\mathbb{R}^3)$ by using the  multiplier theorem by Lizorkin (see \cite{Liz1963}, \cite{Liz1967} or Chapter VII.4 in \cite{Galdi}), which implies that $u\equiv C$.
 Li-Pan \cite{LP} proved two forms of Liouville theorems for D-solutions of the MHD equations when one of  the magnetic field or  the velocity field  tends to a non-zero constant vector at infinity while the other one tends to zero. See also recent results with different viscosity coefficients by Wang-Yang in \cite{WY2024}.

	However, Liouville-type problems for the fractional case (\(s \neq 1\)) have been significantly less explored. Under some mild assumptions over the external force, there exists at
least one solution $u \in \dot{H}^s\left(\mathbb{R}^3\right) $  obtained by Chamorro-Poggi in \cite{CP2025}. Tang-Yu \cite{TY2016} studied partial H\"{o}lder regularity of the steady fractional Navier-Stokes equations with
$\frac12< s < 1$, and proved that if $\frac12 < s < \frac56$ , any steady suitable weak solution is regular away from a
relatively closed set with zero $(5-6s)$-Hausdorff measure and it is regular when $ \frac56\leq s <1$. 
In \cite{WX2018}, Wang-Xiao proved that smooth solutions to \eqref{eq:main} vanish identically when \(u \in \dot{H}^s\left(\mathbb{R}^3\right) \cap L^{\frac{9}{2}}\left(\mathbb{R}^3\right)\) for \(0<s <1\) by Caffarelli-Silvestre’s extension \cite{CS07,CdLM20,CH21}, and later Yang \cite{Yang2022} extended this result for the range \(\frac{5}{6} \leq s<1\), proving that smooth solutions are trivial under the weaker condition \(u \in L^{\frac{9}{2}}\left(\mathbb{R}^3\right)\). We also refer to the recent $L^p$-theory results by Jarrin-Vergara-Hermosilla in \cite{JV2024} and  the corresponding result of \cite{Yang2022} in high dimension  by Liu-Zuo \cite{LZ2025}. 
On the other hand, for \(\frac{3}{10} < s < 1\), the authors of \cite{CP2025} addressed the Liouville type problem under the condition \(u \in \dot{H}^s\left(\mathbb{R}^3\right) \cap L^{p(s)}\left(\mathbb{R}^3\right)\), where the parameter \(p(s)\) depends on \(s\) and approaches the critical value \(\frac{6}{3-2s}\) in a certain sense. 
	More recently, by establishing a new formula for the natural energy $\int_{\mathbb{R}^3} |(-\Delta)^{\frac{s}{2}}u|^2 dx$ of weak solutions, Tan in \cite{Tan2025} obtained the Liouville type theorem for $\frac12<s<1$ by assuming $u\in  \dot{H}^s\left(\mathbb{R}^3\right) \cap \dot{B}^{1-2s}_{\infty,\infty}$. 

It is still unknown that whether the integrability condition of $u$ like $L^{\frac{9}{2}}\left(\mathbb{R}^3\right)$ or  $\dot{B}^{1-2s}_{\infty,\infty}$ could be removed. In this paper we will investigate this issue. When quipped with the uniform condition at spatial infinity:
\begin{eqnarray} \label{eq:the uniform condition at spatial infinit}
	u(x) \rightarrow u_\infty=(1,0,0)^T \quad \text { uniformly as } \quad|x| \rightarrow \infty,
\end{eqnarray}
one of our main results is stated as follows. 
\begin{Thm} \label{thm:main1}
		Let $\frac12<s<1$. Suppose that $u \in \dot{H}^s\left(\mathbb{R}^3\right) \cap \dot{W}^{1+2s,\infty}\left(\mathbb{R}^3\right)$ is a smooth solution of \eqref{eq:main}  with the uniform condition \eqref{eq:the uniform condition at spatial infinit}, then $u\equiv u_\infty$.
	\end{Thm}
It follows from Theorem \ref{thm:main1} immediately that
	\begin{Cor} \label{cor}
		Let $\frac12<s<1$. Suppose that $u \in \dot{H}^s\left(\mathbb{R}^3\right) \cap \dot{W}^{3,\infty}\left(\mathbb{R}^3\right)$ is a smooth solution of \eqref{eq:main}  with the uniform condition \eqref{eq:the uniform condition at spatial infinit}, then $u\equiv u_\infty$.
	\end{Cor}
    \begin{Rem}
    	By the following scaling invariant of fractional Navier-Stokes system:
    	$$
    	u_\lambda(x)=\lambda^{2s-1} u(\lambda x), \quad p_\lambda(x)=\lambda^{4s-2} p(\lambda x),
    	$$
    	we may assume the constant vector $u_{\infty}=(C,0,0)^T$, where $C\neq0$. Moreover, when $s=1$, the derivatives of the velocity and the pressure also converge to the constant in 3D, for example seeing Theorem X.5.1 in \cite{Galdi}. Hence the condition of $u\in \dot{W}^{1+2s,\infty}\left(\mathbb{R}^3\right)$ or $u\in \dot{W}^{3,\infty}\left(\mathbb{R}^3\right)$ seems to be natural.
    \end{Rem}
	
In fact, the proof of Theorem \ref{thm:main1} and Corollary \ref{cor} is based on the following criterion. 
	\begin{Thm} \label{thm:main}
		Let $\frac12<s<1$. Assume that $u \in \dot{H}^s\left(\mathbb{R}^3\right)$ is a smooth solution of \eqref{eq:main}  with the uniform condition \eqref{eq:the uniform condition at spatial infinit} such that
		\begin{eqnarray} \label{eq:D_condition}
			\int_{\mathbb{R}^3} |\nabla u |^{1+\frac1s} dx < \infty,
		\end{eqnarray}
		then $u\equiv u_\infty$.
	\end{Thm}
	\begin{Rem}
		It seems to be difficult to  remove the condition \( \nabla u \in L^{1+1/s}\left(\mathbb{R}^3\right) \). The reason is that when proving \(L^p\) estimates for the perturbed fractional linear Oseen system(see Lemma \ref{lem:WangYang}), an appropriate norm of \(\nabla u\) is required, which cannot be supplied by the natural energy space \( \dot{H}^s\left(\mathbb{R}^3\right) \) to which \(u\) belongs.
	\end{Rem}
	
	This paper is organized as follows: some notations and some necessary lemmas are presented in Sect. \ref{sec2}; In Sect. \ref{sec3}, we prove Theorem \ref{thm:main1} and Corollary \ref{cor} under the criterion Theorem \ref{thm:main}; The Liouville type theorem on the fractional Navier-Stokes system with $u_\infty\neq0$ is obtained in Sect. \ref{sec4}, where we give the detailed proof of Theorem \ref{thm:main}.

	Throughout this paper, $C\left(c_1, c_2, \ldots, c_n\right)$ denotes a positive constant depending on $c_1, c_2, \ldots c_n$ which may be different from line to line.

	\section{Preliminaries} \label{sec2}
	
	Firstly, we introduce some notations. Denote by $B_r\left(x_0\right):=\left\{x \in \mathbb{R}^3:\left|x-x_0\right|<r\right\}$ and $B_r:=B_r(0)$. Denote by $\nabla^\gamma:=\partial_{x_1}^{\gamma_1} \partial_{x_2}^{\gamma_2} \partial_{x_3}^{\gamma_3}$, where  $\gamma=\left(\gamma_1, \gamma_2, \gamma_3\right)$, $\gamma_1, \gamma_2, \gamma_3 \in \mathbb{N} \cup\{0\}$, $\partial_i = \frac{\partial}{\partial x_i}$ and $|\gamma|=\gamma_1+\gamma_2+\gamma_3$. Denote $L^p(\Omega)$ by the usual Lebesgue space with the norm
	$$
	\|f\|_{L^p(\Omega)}:= \begin{cases}\left(\int_{\Omega}|f(x)|^p d x\right)^{1 / p}, & 1 \leq p<\infty, \\ \underset{x \in \Omega}{\operatorname{ess\, sup}}|f(x)|, & p=\infty,\end{cases}
	$$
	where $\Omega \subset \mathbb{R}^3, 1 \leq p \leq \infty$. $W^{k, p}(\Omega)$  and $\dot{W}^{k, p}(\Omega)$ are defined as follows:
	$$
	\begin{aligned}
		\|f\|_{W^{k, p}(\Omega)} & :=\sum_{0 \leq|\gamma| \leq k}\left\|\nabla^\gamma f\right\|_{L^p(\Omega)}, \\
		\|f\|_{\dot{W}^{k, p}(\Omega)} & :=\sum_{|\gamma|=k}\left\|\nabla^\gamma f\right\|_{L^p(\Omega)},
	\end{aligned}
	$$
	respectively. $C^{\infty}(\Omega)$ denotes the space of smooth functions on $\Omega$. $\mathcal{S}\left(\mathbb{R}^n\right)$ denotes the space of rapid decreasing smooth functions on $\mathbb{R}^n$. $\mathcal{P}_n$ stands for the space of polynomials in $\mathbb{R}^3$ with their degree no bigger than $n$. If $0<s<n$, the Riesz potential $I^s f$ of a locally integrable function $f$ on $\mathbb{R}^n$ is the function defined by
	$$
	\left(I^s f\right)(x)=\frac{1}{C(s)} \int_{\mathbb{R}^n} \frac{f(y)}{|x-y|^{n-s}} \mathrm{d} y.
	$$

	 Next, we introduce the commutators in terms of H\"{o}lder norms, namely we consider
	\begin{eqnarray} \label{eq:commutators}
		\left[(-\Delta)^{s}, g\right](f)=(-\Delta)^{s}(g f)-g(-\Delta)^{s} f.
	\end{eqnarray}
	The following lemma is called Coifman-McIntosh-Meyer type commutator estimate, which plays an important role in dealing with the commutator $$(-\Delta)^s\left[\psi(x)v(x)\right]-	\psi(x)(-\Delta)^s v(x) $$
 in the proof of the main result (see \eqref{eq:I_2}, \eqref{eq:I4} and \eqref{eq:I5}, for example).
	\begin{Lem}[Coifman-McIntosh-Meyer type commutator estimate, Theorem 6.1, \cite{LS20}] \label{thm2.2}
Let $s \in(0,1]$ and $p \in(1, \infty)$. Then, for $\sigma \in[s, 1], f, g \in C_c^{\infty}\left(\mathbb{R}^n\right)$, it holds
$$
\left\|\left[(-\Delta)^{\frac{s}{2}}, g\right](f)\right\|_{L^p\left(\mathbb{R}^n\right)} \lesssim[g]_{C^\sigma}\left\|I^{\sigma-s} f\right\|_{L^p\left(\mathbb{R}^n\right)}.
$$
Also, for $q_1, q_2, p \in(1, \infty), \frac{1}{q_1}+\frac{1}{q_2}=\frac{1}{p}, \sigma \in[s, 1)$, it holds
\begin{equation} \label{eq:CoifmanMcIntoshMeyer type commutator estimate}
\left\|\left[(-\Delta)^{\frac{s}{2}}, g\right](f)\right\|_{L^p\left(\mathbb{R}^n\right)} \lesssim\left\|(-\Delta)^{\frac{\sigma}{2}} g\right\|_{L^{q_1}\left(\mathbb{R}^n\right)}\left\|I^{\sigma-s} f\right\|_{L^{q_2}\left(\mathbb{R}^n\right)}.
\end{equation}
	\end{Lem} 
	For $n=1$ and $s=\frac12$, the commutator \eqref{eq:commutators} is also called the first Calderón commutator \cite{Cal65}.
	\begin{Rem}
		In fact, Lemma \ref{thm2.2} remains valid if we replace \( f, g \in C_c^{\infty}(\mathbb{R}^n) \) with \( f, g \in C^{\infty}(\mathbb{R}^n) \) and the boundary terms related to \( f \) and \( g \) vanish in integration by parts, as the compact support condition is only used to eliminate boundary terms in the proof of Lemma \ref{thm2.2}.
	\end{Rem}
	The following lemma concerns the fractional Leibniz rule, providing an alternative approach to handling commutators ( See \cite{GO14} and \cite{NT2019} for example). 
	\begin{Lem}[Fractional Leibniz rule] \label{lem:Fractional Leibniz rule}
		Let $f$, $g$ be two smooth functions in $\mathbb{R}^3$ and $1<p<\infty$. It holds that: \\
		(1) for $s>0$ and $1<p_0, p_1, q_0, q_1 \leq \infty$,
		$$
		\left\|(-\Delta)^{\frac{s}{2}}(f g)\right\|_{L^p\left(\mathbb{R}^3\right)} \leq C\left\|(-\Delta)^{\frac{s}{2}} f\right\|_{L^{p_0}\left(\mathbb{R}^3\right)}\|g\|_{L^{p_1}\left(\mathbb{R}^3\right)}+C\|f\|_{L^{q_0}\left(\mathbb{R}^3\right)}\left\|(-\Delta)^{\frac{s}{2}} g\right\|_{L^{q_1}\left(\mathbb{R}^3\right)},
		$$
		where $\frac{1}{p}=\frac{1}{p_0}+\frac{1}{p_1}=\frac{1}{q_0}+\frac{1}{q_1}$.\\
		(2) for $0<s, s_1, s_2<1< p_1, p_2<\infty$,
		$$
		\left\|(-\Delta)^{\frac{s}{2}}(f g)-(-\Delta)^{\frac{s}{2}}(f) g-(-\Delta)^{\frac{s}{2}}(g) f\right\|_{L^p\left(\mathbb{R}^3\right)} \leq C\left\|(-\Delta)^{\frac{s_1}{2}} f\right\|_{L^{p_1}\left(\mathbb{R}^3\right)}\left\|(-\Delta)^{\frac{s_2}{2}} g\right\|_{L^{p_2}\left(\mathbb{R}^3\right)} ,
		$$
		where $s=s_1+s_2$ and $\frac{1}{p}=\frac{1}{p_1}+\frac{1}{p_2}$.
	\end{Lem}
	Then, we introduce the following multiplier theorem by Lizorkin(see \cite{Liz1963}, \cite{Liz1967} or Chapter VII.4 \cite{Galdi}), which is crucial for the proof of our main results.
	\begin{Thm}[Lizorkin]\label{lem:lizorkin}
		Let $\Phi: \mathbb{R}^n \rightarrow \mathbb{C}$ be continuous together with the derivative
		$$
		\frac{\partial^n \Phi}{\partial \xi_1 \ldots \partial \xi_n}
		$$
		and all preceding derivatives for $\left|\xi_i\right|>0, i=1, \ldots, n$. Then, if for some $\beta \in[0,1)$ and $M>0$
		\begin{equation}
			\label{eq:bound of phi}
			\left|\xi_1\right|^{\kappa_1+\beta} \cdot \ldots \cdot\left|\xi_n\right|^{\kappa_n+\beta}\left|\frac{\partial^\kappa \Phi}{\partial \xi_1^{\kappa_1} \ldots \partial \xi_n^{\kappa_n}}\right| \leq M,
		\end{equation}
		where $\kappa_i$ is zero or one and $\kappa=\sum_{i=1}^n \kappa_i=0,1, \ldots n$, the integral transform
		$$
		T u =\frac{1}{(2 \pi)^{n / 2}} \int_{\mathbb{R}^n} e^{i \boldsymbol{x} \cdot \boldsymbol{\xi}} \Phi(\xi) \widehat{u}(\xi) d \xi, \quad u \in \mathcal{S}\left(\mathbb{R}^n\right),
		$$
		defines a bounded linear operator from $L^q\left(\mathbb{R}^n\right)$ into $L^r\left(\mathbb{R}^n\right), 1<q<$ $\infty, 1 / r=1 / q-\beta$, and we have
		$$
		\|T u\|_{L^r\left(\mathbb{R}^n\right)} \leq C (q, \beta, M)\|u\|_{L^q\left(\mathbb{R}^n\right)}.
		$$
	\end{Thm}
	
	In order to deal with the perturbation system of the generalized Oseen system, we introduce the definition of the relative bounded-ness and Kato's stability theorem of bounded invertibility (see Chapter Four, §1, Theorem 1.16 in \cite{Kato}). 
	\begin{Def}
		Let $\mathcal{T}$ and $\mathcal{A}$ be operators with the same domain space $\mathcal{X}$ (but not necessarily with the same range space) such that $\mathrm{D}(\mathcal{T}) \subset \mathrm{D}(\mathcal{A})$ and
		\begin{equation}
			\label{def2.13}
			\|\mathcal{A} u\| \leq a\|u\|+b\|\mathcal{T} u\|, \quad u \in \mathrm{D}(\mathcal{T}),
		\end{equation}
		where $a, b$ are nonnegative constants. Then we shall say that $\mathcal{A}$ is relatively bounded with respect to $\mathcal{T}$ or simply $\mathcal{T}$-bounded.
	\end{Def}
	\begin{Lem}[Kato]
		\label{lem2.5}
		Assume $\mathcal{X}$ and $\mathcal{Y}$ are Banach spaces. Let $\mathcal{T}$ and $\mathcal{A}$ be operators from $\mathcal{X}$ to $\mathcal{Y}$. Let $\mathcal{T}^{-1}$ exist and belong to $\mathscr{B}(\mathcal{Y}, \mathcal{X})$ ($\mathscr{B}(\mathcal{Y}, \mathcal{X})$ is the set of all bounded operators on $\mathcal{Y}$ to $\mathcal{X} $ and so that $\mathcal{T}$ is closed). Let $\mathcal{A}$ be $\mathcal{T}$-bounded, with the constants $a, b$ in (\ref{def2.13}) satisfying the inequality
		$$
		a\left\|\mathcal{T}^{-1}\right\|+b<1 .
		$$
		Then $\mathcal{S}=\mathcal{T}+\mathcal{A}$ is closed and invertible, and 
		\begin{eqnarray}\label{eq:kato}
			\left\|\mathcal{S}^{-1}\right\|\leq \frac{\left\|\mathcal{T}^{-1}\right\|}{1-a\left\|\mathcal{T}^{-1}\right\|-b}.
		\end{eqnarray}
	\end{Lem}
	
	Then, we consider the 3D stationary linear generalized Ossen system:
	\begin{equation} \label{eq:the 3D stationary linear generalized degenerate Ossen system}
		\left\{\begin{array}{l}
			(-\Delta)^s v+\partial_{1} v+\nabla p=f, \\
			\nabla \cdot v=g.
		\end{array} \right.
	\end{equation} 
	For convenience, we define a function space $H^{s,k}(\mathbb{R}^3)$, $\frac12< s<1$, $1\leq k \leq \infty$, with norm $\|f\|_{H^{s,k}(\mathbb{R}^3)} := \|f\|_{L^k(\mathbb{R}^3)}   +  \|(-\Delta)^{s-1}\nabla f\|_{L^k(\mathbb{R}^3)}$ for any $f\in H^{s,k}(\mathbb{R}^3)$. It is easy to check that $H^{s,k}(\mathbb{R}^3)$ is a Banach space. By Theorem \ref{lem:lizorkin}, the following \(L^p\) estimates for \eqref{eq:the 3D stationary linear generalized degenerate Ossen system} could be obtained.
	\begin{Lem}[\(L^p\) estimates for the fractional linear Oseen system] \label{lem:3D stationary linear generalized degenerate Ossen system}
		Let $f\in L^k\left(\mathbb{R}^3\right)$, $g \in  H^{s,k}(\mathbb{R}^3)$ with $1<k<\frac{s+1}{s}$,$\frac12< s<1$ and $r=\left(\frac{1}{k}-\frac{s}{s+1}\right)^{-1}$. Then for the generalized Ossen system \eqref{eq:the 3D stationary linear generalized degenerate Ossen system}, there exists a unique $(v, p) \in\left(L^r\left(\mathbb{R}^3\right)\right)^3 \times\left(\dot{W}^{1, k}\left(\mathbb{R}^3\right) / \mathcal{P}_0\left(\mathbb{R}^3\right)\right)$ such that
		\begin{equation} \label{eq:temp1}
			\begin{aligned}
				 \|v\|_{L^r\left(\mathbb{R}^3\right)}+\|\nabla p\|_{L^k\left(\mathbb{R}^3\right)} 
				\leq C(k)\left(\left\|f\right\|_{L^k\left(\mathbb{R}^3\right)} +\|g\|_{H^{s,k}(\mathbb{R}^3)}\right) .
			\end{aligned}
		\end{equation}
	\end{Lem}
	\begin{proof}
		By performing a Fourier transform on the system of \eqref{eq:the 3D stationary linear generalized degenerate Ossen system}, we obtain that
		\begin{equation} \label{eq:the 3D stationary linear generalized degenerate Ossen system after a Fourier transform}
			\left\{\begin{array}{l}
				(|\xi|^{2s} +i\xi_1)\hat{v} + i\xi \hat{p}=\hat{f}, \\
				i\xi\cdot \hat{v}=\hat{g},
			\end{array} \qquad \text { in } \mathbb{R}^3.\right.
		\end{equation}
		Taking $i\xi$ on the both sides of  $\eqref{eq:the 3D stationary linear generalized degenerate Ossen system after a Fourier transform}_1$ as the vector inner product, we get
		\begin{eqnarray*}
			(|\xi|^{2s}+i\xi_1)\hat{g} - |\xi|^2 \hat{p}=i\xi\cdot\hat{f},
		\end{eqnarray*}
		from which we solve
		\begin{eqnarray*}
			\hat{p} = \frac{|\xi|^{2s}+i\xi_1}{|\xi|^2}\hat{g} - \frac{i\xi\cdot\hat{f}}{|\xi|^2}
		\end{eqnarray*}
		and get that
		\begin{eqnarray} \label{eq:i xi p}
			i\xi\hat{p} = |\xi|^{2(s-1)}i\xi\hat{g}-\frac{\xi_1\xi}{|\xi|^2}\hat{g} + \frac{\xi\otimes\xi\cdot\hat{f}}{|\xi|^2}.
		\end{eqnarray}
		With $\hat{p}$ in hand, $\hat{v}$ could be solved from $\eqref{eq:the 3D stationary linear generalized degenerate Ossen system after a Fourier transform}_1$: 
		\begin{eqnarray*}
			\hat{v}\m&=&\m \frac{1}{|\xi|^{2s}+i\xi_1}\left(id-\frac{\xi\otimes\xi}{|\xi|^2}\right)\hat{f} 
			- \frac{1}{|\xi|^{2s}+i\xi_1}|\xi|^{2(s-1)}i\xi\hat{g} 
			 + \frac{1}{|\xi|^{2s}+i\xi_1}\frac{\xi_1\xi}{|\xi|^2}\hat{g} \no\\
			&=:&\m \Phi_1\hat{f} + \Phi_2 (|\xi|^{2(s-1)}i\xi\hat{g}) + \Phi_3 \hat{g}.
		\end{eqnarray*}
		By \eqref{eq:i xi p} and Theorem \ref{lem:lizorkin}, we easily infer an estimate of $ p$: 
		\begin{eqnarray} \label{eq:nabla p}
			\|\nabla p\|_{L^k (\mathbb{R}^3)} \leq C\left(\left\|f\right\|_{L^k\left(\mathbb{R}^3\right)}+\|g\|_{H^{s,k}(\mathbb{R}^3)}\right).
		\end{eqnarray}

		For the rest of the proof, it suffices to consider $\Phi_2$, since $\Phi_1$ and $\Phi_3$ are similar. Using Young inequality, there holds
		\begin{eqnarray*}
			|\xi_1|^\frac{s}{s+1} |\xi_2|^\frac{s}{s+1} |\xi_3|^\frac{s}{s+1} |\Phi_2| \leq C \frac{|\xi_1|+|\xi_2|^{2s}+|\xi_3|^{2s}}{\sqrt{|\xi|^{4s}+|\xi_1|^2}} \leq C
		\end{eqnarray*}
		for all $\xi \in \left\{\xi \in \mathbb{R}^3:\left|\xi_i\right|>0\right.,\left.i=1,2,3\right\}$, which proves \eqref{eq:bound of phi} when $\kappa_1=\kappa_2=\kappa_3=0$. The proof for non-zero $\left(\kappa_1, \kappa_2, \kappa_3\right)$ is similar. Therefore, we get
		\begin{equation*}
			\begin{aligned}
				\|v\|_{L^r\left(\mathbb{R}^3\right)} \leq C(k)\left(\left\|f\right\|_{L^k\left(\mathbb{R}^3\right)}+\|g\|_{H^{s,k}(\mathbb{R}^3)}\right),
			\end{aligned}
		\end{equation*}
		which combined \eqref{eq:nabla p} completes the proof.
	\end{proof}

	Next, we show that by Lemma \ref{lem2.5}, \(L^p\) estimates \eqref{eq:temp1} still hold for the following perturbation generalized Oseen system in $\mathbb{R}^3$:
	\begin{equation}
		\label{eq:the perturbation Oseen system:}
		\left\{\begin{array}{l}
			(-\Delta)^s v+\partial_{1} v+\mathcal{M}v+\nabla p=f, \\
			\nabla \cdot v=g.
		\end{array} \right. 
	\end{equation}
	\begin{Lem}[\(L^p\) estimates for the perturbation generalized Oseen system]
		\label{lem:WangYang}
		Assume that $1<k<\frac{s+1}{s}$, $f \in L^k\left(\mathbb{R}^3\right)$, $g \in  H^{s,k}(\mathbb{R}^3)$ and $\mathcal{M} \in$ $\left(L^{\frac{s+1}{s}}\left(\mathbb{R}^3\right)\right)^{3 \times 3}$. Let $(v, p)$ satisfy the perturbation generalized Oseen system \eqref{eq:the perturbation Oseen system:}. Then, there exists a small constant $\varepsilon >0$ such that if
		$$
		\|\mathcal{M}\|_{L^{\frac{s+1}{s}}\left(\mathbb{R}^3\right)}<\varepsilon,
		$$ there exists a unique solution $(v, p) \in\left(L^r\left(\mathbb{R}^3\right)\right)^3 \times\left(\dot{W}^{1, k}\left(\mathbb{R}^3\right) / \mathcal{P}_0\left(\mathbb{R}^3\right)\right)$  such that
		\begin{equation} \label{eq:esimate of v theta p}
				\begin{aligned}
					\|v\|_{L^r\left(\mathbb{R}^3\right)}+\|\nabla p\|_{L^k\left(\mathbb{R}^3\right)} 
					\leq C(k,\varepsilon)\left(\left\|f\right\|_{L^k\left(\mathbb{R}^3\right)}+\|g\|_{H^{s,k}(\mathbb{R}^3)}\right) .
				\end{aligned}
		\end{equation}
		where $r=\left(\frac{1}{k}-\frac{s}{s+1}\right)^{-1}$.
	\end{Lem}
	\begin{proof}
		Denote the Banach spaces
		$$
		\begin{aligned}
			& \mathcal{X}:=\left(L^r\left(\mathbb{R}^3\right)\right)^3 \times \dot{W}^{1, k}\left(\mathbb{R}^3\right) / \mathcal{P}_0\left(\mathbb{R}^3\right), \\
			& \mathcal{Y}:=\left(L^k\left(\mathbb{R}^3\right)\right)^3 \times  H^{s,k}(\mathbb{R}^3) ,
		\end{aligned}
		$$
		with norms
		$$
		\begin{aligned}
			\|(v, p)\|_{\mathcal{X}} & :=\|v\|_{L^r\left(\mathbb{R}^3\right)}+\|\nabla p\|_{L^k\left(\mathbb{R}^3\right)}, \\
			\left\|\left(f,g\right)\right\|_{\mathcal{Y}} & :=\left\|f\right\|_{L^k\left(\mathbb{R}^3\right)} + \|g\|_{H^{s,k}\left(\mathbb{R}^3\right)},
		\end{aligned}
		$$
		respectively. Thanks to Lemma \ref{lem:3D stationary linear generalized degenerate Ossen system}, the operator $\mathcal{K}$ which is defined by
		$$
		\begin{aligned}
			\mathcal{K}:\qquad \mathcal{X} & \mapsto \mathcal{Y} \\
			\left(\begin{array}{l}
				v \\
				p
			\end{array}\right) & \mapsto\left(\begin{array}{l}(-\Delta)^s v+\partial_{1} v+\nabla p\\
				\nabla \cdot v\end{array}\right),
		\end{aligned}
		$$
		admits a bounded inverse operator $\mathcal{K}^{-1}$ and there exists a constant $C_1=C(k)$ such that
		$$
		\|v\|_{L^r\left(\mathbb{R}^3\right)}+\|\nabla p\|_{L^k\left(\mathbb{R}^3\right)} \leq C_1\|\mathcal{K}(v, p)\|_{\mathcal{Y}} .
		$$
		Note that the operator $\mathcal{H}: \mathcal{X} \mapsto \mathcal{Y}$, which is defined by
		$$
		\mathcal{H}(v, p)=\left(\mathcal{M}v, 0\right) ,
		$$
		satisfies
		$$
		\begin{aligned}
			\|\mathcal{H}(v, p)\| _{\mathcal{Y}} \leq \|\mathcal{M}\|_{L^\frac{s+1}{s}\left(\mathbb{R}^3\right)}\|v\|_{L^r\left(\mathbb{R}^3\right)}\leq C_1\|\mathcal{M}\|_{L^\frac{s+1}{s}\left(\mathbb{R}^3\right)}\|\mathcal{K}(v, p)\|_{\mathcal{Y}}.
		\end{aligned}
		$$
		Choosing $\varepsilon = (C_1+1)^{-1}$, it follows that
		$$
		C_1\|\mathcal{M}\|_{L^\frac{s+1}{s}\left(\mathbb{R}^3\right)}<1 .
		$$
		Then Lemma \ref{lem2.5} yields if $\|\mathcal{M}\|_{L^\frac{s+1}{s}\left(\mathbb{R}^3\right)} <\varepsilon$, the operator $\mathcal{K}+\mathcal{H}$ admits a bounded inverse. Thus, the estimate \eqref{eq:esimate of v theta p} holds automatically by the formula \eqref{eq:kato}. The proof is completed.
	\end{proof}
	At the end of this section, we introduce the fractional Gagliardo-Nirenberg inequality (see, e.g., Theorem 2.44 in \cite{BCD11}), which is used to prove Theorem \ref{thm:main1} and Corollary \ref{cor}.
	\begin{Lem}[fractional Gagliardo-Nirenberg inequality] \label{lem:FGN}
		Let $1<q,r\leq\infty$ and $0<\sigma<t<\infty$. There exists a constant $C$ such that
		$$
		\|u\|_{\dot{W}^{\sigma,p}(\mathbb{R}^3)} \leq C\|u\|_{L^q(\mathbb{R}^3)}^\theta\|u\|_{\dot{W}^{t,r}(\mathbb{R}^3)}^{1-\theta} \quad \text { with } \quad \frac{1}{p}=\frac{\theta}{q}+\frac{1-\theta}{r} \quad \text { and } \quad \theta=1-\frac{\sigma}{t}.
		$$
	\end{Lem}
	
	\section{Proof of Theorem \ref{thm:main1} and Corollary \ref{cor}} \label{sec3}
	
	In this section, we prove Theorem \ref{thm:main1} and Corollary \ref{cor} under Theorem \ref{thm:main}.
	
	\begin{proof}[Proof of Theorem \ref{thm:main1}.]
		Using Lemma \ref{lem:FGN} and denoting $\theta':= \frac{2s}{s+1}$, we have
		\begin{equation*}\begin{aligned}
				\|\nabla u\|_{L^{1+\frac1s}(\mathbb{R}^3)} \leq C\|u\|_{\dot{H}^{s}(\mathbb{R}^3)}^{\theta'} \| u\|_{\dot{W}^{1+2s, \infty}(\mathbb{R}^3)}^{1-\theta'},
		\end{aligned}\end{equation*}
		for $\frac12<s<1$. Then, combining Theorem \ref{thm:main}, the proof is completed.
	\end{proof}
	
	\begin{proof}[Proof of Corollary \ref{cor}.]
		Thanks to Lemma \ref{lem:FGN}, denoting $\theta'':= \frac{2-2s}{3}$, we get
		\begin{equation*}\begin{aligned}
				\| u\|_{\dot{W}^{1+2s,\infty}(\mathbb{R}^3)}\leq C\|u\|_{L^\infty(\mathbb{R}^3)}^{\theta''} \| u\|_{\dot{W}^{3, \infty}(\mathbb{R}^3)}^{1-\theta''},
		\end{aligned}\end{equation*}
		for $\frac12<s<1$. Thus, by Theorem \ref{thm:main1}, the proof is complete.
	\end{proof}
	
	\section{Proof of Theorem \ref{thm:main}} \label{sec4}
	\begin{proof}[Proof of Theorem \ref{thm:main}.]
		We start the proof by the standard substitution of variables. Denoting $v:=u-(1,0,0)^T$, it follows from \eqref{eq:main} that 
		\begin{equation} 
			\label{eq:v}
			\left\{\begin{array}{l}
				(- \Delta)^s v+v\cdot\nabla v +\partial_{1} v+\nabla p=0, \\
				\nabla \cdot v=0,
			\end{array} \quad \quad \text { in } \mathbb{R}^3.\right.
		\end{equation}
		In the following, we focus on the generalized fractional Ossen system \eqref{eq:v} and divide the rest of proof into three steps.
		
		{\bf Step I: the estimate of pressure.} 
		By the condition \eqref{eq:D_condition} and Sobolev embedding theorem, we get
		\begin{equation} \label{eq:L6-estimate}
			\|v\|_{L^{\frac{3s+3}{2s-1}}(\mathbb{R}^3)}  \leq C\|\nabla v\|_{L^{\frac{s+1}{s}}(\mathbb{R}^3)} < \infty.
		\end{equation}
		Then using H${\rm \ddot{o}}$lder inequality and the continuity of the operator $\mathcal{R}_i$ on Lebesgue space $L^\ell\left(\mathbb{R}^3\right)$ for $2<\ell<+\infty$, it holds that 
		\begin{equation}
			\label{eq:vivj}
			\|\Delta^{-1}\partial_i \partial_j\left(v_i v_j\right)\|_{L^{\frac{3s+3}{4s-2}}(\mathbb{R}^3)} \leq C\|v_i v_j\|_{L^{\frac{3s+3}{4s-2}}(\mathbb{R}^3)} \leq C\|v\|^2_{L^{\frac{3s+3}{2s-1}}(\mathbb{R}^3)} < \infty,
		\end{equation}
		for $i,j=1,2,3$.
		Acting the divergence operator on $(\ref{eq:v})_1$ and combining \eqref{eq:vivj}, one has
		$$
		\begin{aligned}
			p =\sum_{i, j=1}^3-\Delta^{-1} \partial_i \partial_j\left(v_i v_j\right)\in L^{\frac{3s+3}{4s-2}}(\mathbb{R}^3).
		\end{aligned}
		$$
		
		{\bf Step II: $L^3$-estimate of $v$.} 
		According to the condition \eqref{eq:D_condition}, there exists an $M>0$ which is large enough such that
		\begin{eqnarray} \label{eq:M}
			\int_{B_M^c} |\nabla v|^{\frac{s+1}{s}}  ~d x<\varepsilon^{1+\frac1s},
		\end{eqnarray}
		where $B_M^c$ is the complement of $B_M$ in $\mathbb{R}^3$. Define $\psi \in C^{\infty}\left(\mathbb{R}^3\right)$ the cut-off function such that
		\begin{eqnarray} \label{eq:definition of psi}
			\psi(x)=\psi(|x|)=\left\{\begin{array}{lll}
				1, & \text { if } & |x|>2 M ; \\
				0, & \text { if } & |x|<M,
			\end{array}\right.
		\end{eqnarray}
		and $0 \leq \psi(x) \leq 1$ for any $ x \in B_{2 M}\setminus B_M$. Multiplying $\psi$ on the both sides of \eqref{eq:v}, it follows that
		\begin{equation}
			\label{eq:v'}
			\left\{\begin{array}{l}
				(- \Delta)^s (\psi v)+\psi v\cdot(\nabla v) \chi_{B_M^c} +\partial_{1} (\psi v)+\nabla (\psi p)=F(\psi), \\
				\nabla \cdot (\psi v)=\nabla\psi\cdot v,
			\end{array} \quad \quad \text { in } \mathbb{R}^3,\right.
		\end{equation}
		where
		$$
		\begin{aligned}
			& F(\psi)= (-\Delta)^s (\psi v) - \psi(-\Delta)^s v + v\partial_1 \psi + p\nabla\psi.
		\end{aligned}
		$$
		Firstly, we consider the commutators. Since $\frac12<s<1$, there holds 
			\begin{eqnarray} \label{eq: I_1+I_2}
				&&(-\Delta)^s\left[\psi(x)v(x)\right]-	\psi(x)(-\Delta)^s v(x) \no\\
				&=& (-\Delta)^{s-1/2} (-\Delta)^{1/2} \left[\psi(x)v(x)\right] - \psi(x) (-\Delta)^{s-1/2} \left[(-\Delta)^{1/2}v(x)\right] \no\\
				&=& (-\Delta)^{s-1/2} (-\Delta)^{1/2} \left[\psi(x)v(x)\right] - (-\Delta)^{s-1/2}\left[\left((-\Delta)^{1/2}v(x)\right)\psi(x)\right] \no\\
				&&+ (-\Delta)^{s-1/2}\left[\left((-\Delta)^{1/2}v(x)\right)\psi(x)\right] - \psi(x) (-\Delta)^{s-1/2} \left[(-\Delta)^{1/2}v(x)\right] \no\\
				&=:& I_1+I_2,
			\end{eqnarray}
			where $$I_1 := (-\Delta)^{s-1/2} (-\Delta)^{1/2} \left[\psi(x)v(x)\right] - (-\Delta)^{s-1/2}\left[\left((-\Delta)^{1/2}v(x)\right)\psi(x)\right] $$ and  $$I_2 := (-\Delta)^{s-1/2}\left[\left((-\Delta)^{1/2}v(x)\right)\psi(x)\right] - \psi(x) (-\Delta)^{s-1/2} \left[(-\Delta)^{1/2}v(x)\right].$$
		Denoting $k:=\frac{3(s+1)}{4s+1}$, using H\"{o}lder inequality and  \eqref{eq:CoifmanMcIntoshMeyer type commutator estimate}, we have
		\begin{eqnarray} \label{eq:I_2}
			\|I_2\|_{L^{k}(\mathbb{R}^3)} &\leq& C(s)\|(-\Delta)^{s-1/2}\psi(x)\|_{L^{3}(\mathbb{R}^3)}\|(-\Delta)^{1/2}v(x)\|_{L^{\frac{s+1}{s}}\left(\mathbb{R}^3\right)} \no\\
			&\leq& C(s,M) \|\nabla v\|_{L^{\frac{s+1}{s}}\left(\mathbb{R}^3\right)} < \infty.
		\end{eqnarray}
		For $I_1$, it follows from Hardy-Littlewood-Sobolev theorem of fractional integration that
		\begin{eqnarray} \label{eq:I_1}
			\|I_1\|_{L^{k}(\mathbb{R}^3)} 
			&\leq& \left\|(-\Delta)^{s-1}(-\Delta)^{1/2}\left[(-\Delta)^{1/2} \left(\psi(x)v(x)\right)-  \left((-\Delta)^{1/2}v(x)\right)\psi(x)\right]\right\|_{L^{k}(\mathbb{R}^3)} \no\\
			&\leq& C(s)\left\|(-\Delta)^{1/2}\left[(-\Delta)^{1/2} \left(\psi(x)v(x)\right)-  \left((-\Delta)^{1/2}v(x)\right)\psi(x)\right]\right\|_{L^{p}(\mathbb{R}^3)} \no\\
			&\leq& C(s)\left\|\nabla\left[(-\Delta)^{1/2} \left(\psi(x)v(x)\right)-  \left((-\Delta)^{1/2}v(x)\right)\psi(x)\right]\right\|_{L^{p}(\mathbb{R}^3)} \no\\
			&\leq& C(s) \left(\|I_4\|_{L^{p}(\mathbb{R}^3)}+\|I_5\|_{L^{p}(\mathbb{R}^3)}\right),
		\end{eqnarray}
		where 
		$$
		I_4 := (-\Delta)^{1/2}\left(\nabla \psi(x) \otimes v(x)\right) - (-\Delta)^{1/2} v(x) \otimes \nabla \psi(x),
		$$
		$$
		I_5 := (-\Delta)^{1/2}\left( \psi(x) \nabla v(x)\right) - \psi(x) (-\Delta)^{1/2} \nabla v(x)
		$$
		and
		\begin{equation}\begin{aligned} \label{eq:p}
				\frac{1}{k}=\frac{1}{p}-\frac{2-2s}{3}.
		\end{aligned}\end{equation}
		Using H\"{o}lder inequality and \eqref{eq:CoifmanMcIntoshMeyer type commutator estimate} again, one gets
		\begin{eqnarray} \label{eq:I4}
			\|I_4\|_{L^{p}(\mathbb{R}^3)} \leq C \|(-\Delta)^{1/2}\nabla\psi\|_{L^{m_1}(\mathbb{R}^3)} \|v\|_{L^{\frac{3s+3}{2s-1}} (\mathbb{R}^3)} \leq C(M) < \infty
		\end{eqnarray}
		and
		\begin{eqnarray} \label{eq:I5}
			\|I_5\|_{L^{p}(\mathbb{R}^3)} \leq C \|(-\Delta)^{1/2}\psi\|_{L^{m_2}(\mathbb{R}^3)} \|\nabla v\|_{L^{\frac{1+s}{s}} (\mathbb{R}^3)} \leq C(M) < \infty,
		\end{eqnarray}
		where
		\begin{equation}\begin{aligned} \label{eq:m1m2}
				\frac{1}{p}=\frac1m_1+\frac{2s-1}{3s+3},\quad \frac1p=\frac1m_2 + \frac{s}{s+1}.
		\end{aligned}\end{equation}
		Summing up \eqref{eq: I_1+I_2}, \eqref{eq:I_2}, \eqref{eq:I_1}, \eqref{eq:I4} and \eqref{eq:I5}, we get
		\begin{eqnarray*} 
			\left\|(-\Delta)^s\left[\psi(x)v(x)\right]-	\psi(x)(-\Delta)^s v(x)\right\|_{L^{k}(\mathbb{R}^3)} \leq C(M,s) < \infty.
		\end{eqnarray*}
		Then, by H\"{o}lder inequality, it holds
		\begin{eqnarray} \label{eq:F psi}
			\|F(\psi)\|_{L^{k}(\mathbb{R}^3)} \m&\leq&\m C\left(\|(-\Delta)^s\left[\psi(x)v(x)\right]-	\psi(x)(-\Delta)^s v(x) \|_{L^{k}(\mathbb{R}^3)} \right. \no\\ 
			&& \m\left. +\|\partial_{1} \psi\|_{L^{\frac32}(\mathbb{R}^3)}\|v\|_{L^{\frac{3s+3}{2s-1}}(\mathbb{R}^3)}  + \|\nabla\psi\|_{L^{s+1}(\mathbb{R}^3)}\|p\|_{L^{\frac{3s+3}{4s-2}}(\mathbb{R}^3)}\right) \no\\ 
			&\leq& \m C(M,s) < \infty. 
		\end{eqnarray}
		Similarly, 
		\begin{eqnarray} \label{eq:nabla psi v}
			\|\nabla \psi \cdot v\|_{L^k(\mathbb{R}^3)} \m&\leq&\m\|\nabla \psi\|_{L^{\frac{3}{2}}\left(\mathbb{R}^3\right)}\|v\|_{L^{\frac{3s+3}{2s-1}}\left(\mathbb{R}^3\right)} \leq  C(M)<\infty . 
		\end{eqnarray}
		By Hardy-Littlewood-Sobolev theorem of fractional integration and H\"{o}lder inequality, we have 
		\begin{eqnarray} \label{eq:nabla nabla psi v}
			\|(-\Delta)^{s-1}\nabla(\nabla \psi \cdot v)\|_{L^k(\mathbb{R}^3)} \m&\leq&\m C(s)\|\nabla(\nabla \psi \cdot v)\|_{L^p(\mathbb{R}^3)} \nonumber\\
			& \leq&\m C(s)\left(\|\nabla^2\psi\|_{L^{m_1}(\mathbb{R}^3)} \|v\|_{L^{\frac{3s+3}{2s-1}}(\mathbb{R}^3) } + \|\nabla\psi\|_{L^{m_2}(\mathbb{R}^3)}\|\nabla v\|_{L^{\frac{s+1}{s}}(\mathbb{R}^3)}\right)\no\\
			 &\leq& \m C(M,s) < \infty,
		\end{eqnarray}
		where the definitions of $p$, $m_1$ and $m_2$ can be found in \eqref{eq:p} and \eqref{eq:m1m2}.
		Combining \eqref{eq:M}, \eqref{eq:F psi}, \eqref{eq:nabla psi v} and \eqref{eq:nabla nabla psi v}, applying the Lemma \ref{lem:WangYang} to the system \eqref{eq:v'}, we obtain that
		$$
		\|\psi v\|_{L^3(\mathbb{R}^3)}<\infty.
		$$
		Therefore, it follows that $v \in L^3(\mathbb{R}^3)$ since $v$ is smooth.
		Then by the same process as Step I, it is not hard to find $p\in L^{\frac32}(\mathbb{R}^3)$.

		{\bf Step III: the vanishing of $v$.}  
		We introduce the cut-off function $\varphi \in C_c^{\infty}\left(\mathbb{R}^3\right)$ s.t.
		\begin{eqnarray*} 
			\varphi(x)=\varphi(|x|)=\left\{\begin{array}{lll}
				1, & \text { if } & |x|<1; \\
				0, & \text { if } & |x|>2,
			\end{array}\right.
		\end{eqnarray*}
		and $0 \leq \varphi(x) \leq 1$ for any $1 \leq|x| \leq 2$. 
		Then for each $R>1$, denote by
		\begin{eqnarray} \label{eq:definition of varphi}
			\varphi_R(x):=\varphi\left(\frac{|x|}{R}\right).
		\end{eqnarray}
		Multiplying the equation $\eqref{eq:v}_1$ by $ v\varphi_R$ and integrating over $\mathbb{R}^3$, there holds
		\begin{equation}\begin{aligned} \label{eq:Multiplyed equation}
				\int_{\mathbb{R}^3} (-\Delta)^{s} {v} \cdot\left(v\varphi_R \right)dx + \int_{\mathbb{R}^3} {v} \cdot {\nabla} {v} \cdot\left(v\varphi_R \right) dx + \int_{\mathbb{R}^3} \partial_1 v \cdot v\varphi_R dx + \int_{\mathbb{R}^3} {\nabla} p \cdot\left(v\varphi_R \right) d x=0.
		\end{aligned}\end{equation}
		For the first term in the left hand of \eqref{eq:Multiplyed equation}, using the equivalent properties of the operator $(-\Delta)^{s}$ (see, e.g., \cite{Kw2017}), we have
		\begin{equation}\begin{aligned} \label{eq:the first term in the left hand}
				&\int_{\mathbb{R}^3}(-\Delta)^{s} {v} \cdot\left(v\varphi_R \right) d x \\
				 =&\int_{\mathbb{R}^3}(-\Delta)^{s/2} {v} \cdot(-\Delta)^{s/2}\left(v\varphi_R \right) d x \\
				 =&\int_{\mathbb{R}^3}(-\Delta)^{s/2} v \cdot\left[(-\Delta)^{s/2}\left(v\varphi_R \right)  +    \varphi_R(-\Delta)^{s/2} v  -   \varphi_R(-\Delta)^{s/2} v\right] d x\\
				=& \int_{\mathbb{R}^3}\left| (-\Delta)^{s/2} v \right|^2\varphi_Rdx + \int_{\mathbb{R}^3} (-\Delta)^{s/2} v \cdot\left[(-\Delta)^{s/2}\left(v\varphi_R \right)   -   \varphi_R(-\Delta)^{s/2} v\right] d x.
		\end{aligned}\end{equation}
	    For the rest of the term in the left hand of \eqref{eq:Multiplyed equation}, by integration by parts, we get
	    \begin{equation}\begin{aligned} \label{eq:the rest of the term in the left hand}
	    		&\int_{\mathbb{R}^3} {v} \cdot {\nabla} {v} \cdot\left(v\varphi_R \right) dx + \int_{\mathbb{R}^3} \partial_1 v \cdot v\varphi_R dx + \int_{\mathbb{R}^3} {\nabla} p \cdot\left(v\varphi_R \right) d x\\
	    		=&-\frac12 \int_{\mathbb{R}^3} |v|^2 v\cdot\nabla\varphi_Rdx - \frac12\int_{\mathbb{R}^3} |v|^2 \partial_1 \varphi_Rdx - \int_{\mathbb{R}^3} pv\cdot\nabla\varphi_Rdx.
	    \end{aligned}\end{equation}
	    Combining \eqref{eq:Multiplyed equation}, \eqref{eq:the first term in the left hand} and \eqref{eq:the rest of the term in the left hand}, we obtain
	    \begin{equation}\begin{aligned}
	    		 \int_{\mathbb{R}^3}\left| (-\Delta)^{s/2} v \right|^2\varphi_Rdx =& \int_{\mathbb{R}^3} (-\Delta)^{s/2} v \cdot\left[ \varphi_R(-\Delta)^{s/2} v- (-\Delta)^{s/2}\left(v\varphi_R \right) \right] d x \\
	    		 &+\frac12 \int_{\mathbb{R}^3} |v|^2 v\cdot\nabla\varphi_Rdx + \frac12\int_{\mathbb{R}^3} |v|^2 \partial_1 \varphi_Rdx + \int_{\mathbb{R}^3} pv\cdot\nabla\varphi_Rdx\\
	    		 =:& I + II + III + IV.
	    \end{aligned}\end{equation}
        Next, we estimate the terms of $I-IV$. For $I$, by H\"{o}lder inequality and Lemma \ref{lem:Fractional Leibniz rule}, one gets
        \begin{equation}\begin{aligned} \label{eq:I}
        		I & =\int_{\mathbb{R}^3} (-\Delta)^{s/2} v \cdot\left[ \varphi_R(-\Delta)^{s/2} v- (-\Delta)^{s/2}\left(v\varphi_R \right) \right] d x\\
        		& \leq\left\|(-\Delta)^{s/2} v\right\|_{L^2(\mathbb{R}^3)}\left\|\varphi_R(-\Delta)^{s/2} v -(-\Delta)^{s/2}\left(v\varphi_R\right)\right\|_{L^2(\mathbb{R}^3)} \\
        		& \leq\|v\|_{\dot{H}^{s}(\mathbb{R}^3)}\left(\left\|(-\Delta)^{s/2}\left(v\varphi_R \right)-  \varphi_R(-\Delta)^{s/2} v -v(-\Delta)^{s/2} \varphi_R \right\|_{L^2(\mathbb{R}^3)}+\left\|v(-\Delta)^{s/2} \varphi_R\right\|_{L^2(\mathbb{R}^3)}\right) \\
        		& \leq\|v\|_{\dot{H}^{s}(\mathbb{R}^3)}\left( \left\|(-\Delta)^{s(1-\theta)/2}\varphi_R\right\|_{L^{\frac{6}{1-\theta}}(\mathbb{R}^3)}   \left\|(-\Delta)^{{s}\theta/2}v\right\|_{L^{\frac{6}{\theta+2}}(\mathbb{R}^3)}   +   \left\|v(-\Delta)^{s/2} \varphi_R\right\|_{L^2(\mathbb{R}^3)}\right) ,
        \end{aligned}\end{equation}
        where $0<\theta<1$.
        Then by the H\"{o}lder inequality, it holds
        \begin{equation}\begin{aligned} \label{eq:temp11}
        		\left\|v(-\Delta)^{s/2} \varphi_R\right\|_{L^2(\mathbb{R}^3)} \leq \|v\|_{L^3(\mathbb{R}^3)}\left\|(-\Delta)^{s/2} \varphi_R\right\|_{L^6(\mathbb{R}^3)}\leq CR^{\frac12-s}\|v\|_{L^3(\mathbb{R}^3)}.
        \end{aligned}\end{equation}
        Thanks to the interpolation theorem(see Theorem 6.4.5 in \cite{BL76} for example) with $s\theta = s\theta + 0\cdot(1-\theta)$ and $\frac{\theta+2}{6} = \frac{\theta}{2} + \frac{1-\theta}{3}$, we have
        \begin{equation}\begin{aligned}
        		\left\|(-\Delta)^{{s}\theta/2}v\right\|_{L^{\frac{6}{\theta+2}}(\mathbb{R}^3)} = \|v\|_{\dot{W}^{s\theta, {\frac{6}{\theta+2}}}(\mathbb{R}^3)} \leq C\|v\|_{\dot{H}^{s}(\mathbb{R}^3)}^\theta \|v\|_{L^3(\mathbb{R}^3)}^{1-\theta},
        \end{aligned}\end{equation}
        and then
        \begin{equation}\begin{aligned} \label{eq:temp12}
        		\left\|(-\Delta)^{s(1-\theta)/2}\varphi_R\right\|_{L^{\frac{6}{1-\theta}}(\mathbb{R}^3)}   \left\|(-\Delta)^{{s}\theta/2}v\right\|_{L^{\frac{6}{\theta+2}}(\mathbb{R}^3)} \leq CR^{-(s-\frac12)(1-\theta)}\|v\|_{\dot{H}^{s}(\mathbb{R}^3)}^\theta \|v\|_{L^3(\mathbb{R}^3)}^{1-\theta}.
        \end{aligned}\end{equation}
        Combining \eqref{eq:I}, \eqref{eq:temp11} and \eqref{eq:temp12}, we obtain
        \begin{equation*}\begin{aligned}
        		I\leq CR^{-(s-\frac12)(1-\theta)}\|v\|_{\dot{H}^{s}(\mathbb{R}^3)}^{1+\theta} \|v\|_{L^3(\mathbb{R}^3)}^{1-\theta} + CR^{\frac12-s}\|v\|_{L^3(\mathbb{R}^3)}\|v\|_{\dot{H}^{s}(\mathbb{R}^3)}.
        \end{aligned}\end{equation*}
        For $II$, $III$ and $IV$, using H\"{o}lder inequality, it holds
        \begin{eqnarray*}
        II=\frac12	\int_{\mathbb{R}^3} \left(v\cdot\nabla\varphi_R\right)|v|^2 dx &\leq& C\|\nabla\varphi_R\|_{L^\infty(\mathbb{R}^3)} \int_{B_{2 R} \backslash B_R} |v|^3 dx \leq CR^{-1}\|v\|^3_{L^3(B_{2 R} \backslash B_R)}, \no\\
       III=\frac12 \int_{\mathbb{R}^3} |v|^2\partial_1\varphi_R~ dx 
        &\leq& C\|\nabla\varphi_R\|_{L^\infty(\mathbb{R}^3)} \left(\int_{B_{2 R} \backslash B_R} |v|^3 dx\right)^{2/3}\left(\int_{B_{2 R} \backslash B_R} ~ dx\right)^{1/3} \no\\
        &\leq& C\|v\|^2_{L^3(B_{2 R} \backslash B_R)},\no\\
        IV=	\int_{\mathbb{R}^3} \left(v\cdot\nabla\varphi_R\right)p~ dx &\leq& \|\nabla\varphi_R\|_{L^\infty(\mathbb{R}^3)} \left(\int_{B_{2 R} \backslash B_R} |v|^3 dx\right)^{1/3}\left(\int_{B_{2 R} \backslash B_R} |p|^{\frac32} dx\right)^{2/3} \no\\
        	&\leq& CR^{-1}\|v\|_{L^3(B_{2 R} \backslash B_R)}\|p\|_{L^{\frac32}(B_{2 R} \backslash B_R)}.
        \end{eqnarray*}
        Due to $v\in L^3(\mathbb{R}^3)\cap\dot{H}^{s}(\mathbb{R}^3)$ and $p\in L^{\frac32}(\mathbb{R}^3)$, letting $R\rightarrow\infty$, we have
        \begin{equation*}\begin{aligned}
        		\lim_{R \rightarrow \infty} I+II+III+IV = 0,
        \end{aligned}\end{equation*}
        which implies
        \begin{equation*}\begin{aligned}
        		\int_{\mathbb{R}^3}\left| (-\Delta)^{s/2} v \right|^2dx = 0.
        \end{aligned}\end{equation*}
       Then by Sobolev embedding theorem, there holds 
       \begin{equation*}\begin{aligned}
       		\|v\|_{L^{\frac{6}{3-2s}}(\mathbb{R}^3)} \leq C\|v\|_{\dot{H}^{s}(\mathbb{R}^3)} = C\left(\int_{\mathbb{R}^3}\left| (-\Delta)^{s/2} v \right|^2dx\right)^{\frac12} = 0,
       \end{aligned}\end{equation*}
       from which we deduce that $v\equiv0$ a.e. in $\mathbb{R}^3$. The proof is completed.
	   \end{proof}
	   \begin{Rem}
	   	In fact, with \( v \in L^3(\mathbb{R}^3) \), the vanishing of \( v \) could also be proved by using the Caffarelli-Silvestre extension(see, for example, Yang's approach in \cite{Yang2022}). 
	   \end{Rem}

	   \noindent {\bf Acknowledgments.}
	   W. Wang was supported by National Key R\&D Program of China (No. 2023YFA1009200), NSFC under grants 12071054 and 12471219.\\
	   
	   \noindent {\bf Declaration of competing interest.}
	   The authors state that there is no conflict of interest.\\
	   
	   \noindent {\bf Data availability.}
	   No data was used in this paper.

\end{document}